\documentclass[letterpaper, 10 pt, conference]{ieeeconf}  

\IEEEoverridecommandlockouts                              

\usepackage{graphics}       
\usepackage{graphicx}
\usepackage{epsfig}         
\usepackage{times}          
\usepackage{amsmath}        
\usepackage{amssymb}        
\usepackage{color}
\usepackage{array}

\usepackage{lipsum}
\usepackage{mathtools, cuted}

\title{\LARGE \bf Regularization of non-overshooting quasi-continuous sliding mode control for chattering suppression at equilibrium}

\author{Michael Ruderman and Denis Efimov  
\thanks{M. Ruderman is with Department of Engineering sciences, University of Agder (UiA), Norway. He is on sabbatical at Polytechnic University of Bari. \newline
Email: {\tt\small michael.ruderman@uia.no}}
\thanks{D. Efimov is with INRIA, Univ. Lille, CNRS, CRIStAL, Lille, France.}
\thanks{The work is partially supported by Aurora program
(RCN grant 340782).}
}

\begin{document}

\newtheorem{theorem}{Theorem}
\newtheorem{prop}{Proposition}
\newtheorem{rem}{Remark}
\newtheorem{defin}{Definition}

\maketitle \thispagestyle{empty} \pagestyle{empty}

\bstctlcite{references:BSTcontrol}

\begin{abstract}

Robust finite-time feedback controller introduced for the second-order systems in \cite{ruderman2024robust} can be seen as a non-overshooting quasi-continuous sliding mode control. The paper proposes a regularization scheme to suppress inherent chattering due to discontinuity of the control \cite{ruderman2024robust} in the origin, in favor of practical applications. A detailed analysis with ISS and iISS proofs are provided along with supporting numerical results.   
\end{abstract}

\section{Introduction}
\label{sec:1}

Robust feedback control systems usually require to meet some criteria in relation to the system and controller transfer characteristics and disturbances upper bound, that is assumed to be known, cf. e.g. \cite{kwakernaak1993robust}. High-gains appear to be a natural way to suppress unknown disturbances with less information required by the control design, see e.g. \cite{tsypkin1999} for minimum phase systems. For static disturbances, an integral control (often resulting in a standard PID or PID-like feedback regulation \cite{aastrom2006}) can be sufficient, see also \cite{khalil2000universal} for nonlinear systems. At the same time, an integral feedback action increases inherently the system order and can also lead to the wind-up effects and even destabilization, cf. \cite{hippe2006}. 

Alternatively to high-gain feedback regulation (e.g. \cite{ilchmann2008} and references therein)
the sliding mode control methods, see \cite{shtessel2014,utkin2020}, become widespread for
compensating the unknown matched perturbations. This is not surprising since a control proportional to the sign of the regulation error is particularly efficient, as already recognized in \cite{fuller1960,haimo1986} for a time-optimal stabilization and in \cite{Tsypkin1984relay} for disturbance compensation. At the same time, a discontinuous control signal in the sliding mode methods can be undesirable for applications, especially with regard to the wear effects, overloading of the actuators, induced parasitic noise, and energy consumption. To this end, in \cite{levant2005} the class of quasi-continuous high order sliding mode control algorithms has been introduced, which have the discontinuity at the origin only. However, this control approach cannot guarantee a monotonic convergence without overshoots in presence of non-vanishing bounded disturbances, and the chattering effect in the origin persists. It appears also to be a common sense that some chattering, which is due to non-modeled dynamics (of actuators and/or sensors in the loop) is unavoidable along with high feedback gain or discontinuous control actions, see e.g. \cite{utkin2015}. The recently proposed novel nonlinear state feedback control for the perturbed second-order systems \cite{ruderman2024robust}, which builds up on the idea of nonlinear damping inverse to the output distance to the origin proposed in \cite{ruderman2021a}, can be seen as a non-overshooting quasi-continuous sliding mode control. 

This work is a continuation of the developments in \cite{ruderman2024robust}, with the aim to overcome discontinuity related issues in the origin, in favor of practical applications. For this purpose, two regularization schemes are proposed for the control algorithm from \cite{ruderman2024robust}, which admit a simple discretization.
The rest of the paper is organized as follows. After introducing the problem statement and used notations, we provide the necessary preliminaries in section \ref{sec:prelim}. The main results of regularization of the control \cite{ruderman2024robust} and the related analysis are given in section \ref{sec:33}. An alternative regularization scheme is also shown. Additional comparative numerical results of the control \cite{ruderman2024robust} and both regularization schemes are visualized in section \ref{sec:3}. The work is concluded by section \ref{sec:4}.

\subsection*{Problem statement}

The closed-loop control system \cite{ruderman2024robust} takes the form:
\begin{gather}
\dot{x}_1(t)=x_{2}(t),\label{eq:CLoriginal}\\
\dot{x}_2(t)=\qquad\;\;\nonumber \\
\begin{cases}
-\bigl|x_{1}(t)\bigr|^{-1}\Bigl(\gamma x_{1}(t)+\bigl|x_{2}(t)\bigr|x_{2}(t)\Bigr)+d(t) & x_{1}(t)\ne0\\
-\gamma\bigl|x_{1}(t)\bigr|^{-1}x_{1}(t)+d(t) & x_{1}(t)=0
\end{cases},\nonumber
\end{gather}
where $\gamma>0$ is the design parameter. If the upper bounded and matched unknown disturbance 
$\|d\|_{\infty}\leq D$ (with the given $D>0$) is non-vanishing, then the sliding mode appears
in the origin $(x_1,x_2)=0$, which is uniformly globally finite-time stable \cite{ruderman2024robust}, due to the sign switching control part of \eqref{eq:CLoriginal}. This can, like any continuously switching control action, lead to unnecessary overloading of the actuator and plant structure, higher energy consumption, as well as wear and noise effects of the controlled process at large.     

Our current goal is to provide a suitable regularization of \eqref{eq:CLoriginal}, while preserving all essential convergence and stability properties of the control system in presence of the exogenous disturbance input and, at the same time, suppressing chattering due to the sliding mode in the stable origin.       

\subsection*{Notation}
\begin{itemize}
\item $\mathbb{R}_{+}=\{x\in\mathbb{R}:x\geq0\}$, where $\mathbb{R}$ is
the set of real numbers.
\item $|\cdot|$ denotes the absolute value in $\mathbb{R}$, $\Vert\cdot\Vert$
is used for the Euclidean norm on $\mathbb{R}^{n}$.
\item For a (Lebesgue) measurable function $d:\mathbb{R}_{+}\to\mathbb{R}^{m}$
and $[t_{0},t_{1})\subset\mathbb{R}_{+}$, we define the norm $\|d\|_{[t_{0},t_{1})}=\text{ess\ sup}_{t\in[t_{0},t_{1})}\Vert d(t)\Vert$.
Then, $\|d\|_{\infty}=\|d\|_{[0,+\infty)}$ and the set of such functions
$d$ satisfying the property $\|d\|_{\infty}<+\infty$ is further
denoted as $\mathcal{L}_{\infty}^{m}$ (the set of essentially bounded
measurable functions).
\item A continuous function $\alpha:\mathbb{R}_{+}\rightarrow\mathbb{R}_{+}$
belongs to the class $\mathcal{K}$ if $\alpha(0)=0$ and it is strictly
increasing. A function $\alpha:\mathbb{R}_{+}\rightarrow\mathbb{R}_{+}$
belongs to the class $\mathcal{K}_{\infty}$ if $\alpha\in\mathcal{K}$
and it is increasing to infinity. A continuous function $\beta:\mathbb{R}_{+}\times\mathbb{R}_{+}\to\mathbb{R}_{+}$
belongs to the class $\mathcal{KL}$ if $\beta(\cdot,t)\in\mathcal{K}$
for each fixed $t\in\mathbb{R}_{+}$ and $\beta(s,\cdot)$ is decreasing
to zero for each fixed $s>0$.
\item The Young's inequality claims that for any $\mathfrak{a},\mathfrak{b}\in\mathbb{R}_{+}$:
\[
\mathfrak{a}\mathfrak{b}\leq\frac{1}{p}\mathfrak{a}^{p}+\frac{p-1}{p}\mathfrak{b}^{\frac{p}{p-1}} \quad \hbox{ for any } p>1.
\]
\end{itemize}

\section{Preliminaries}
\label{sec:prelim}

Consider a nonlinear system:
\begin{equation}
\dot{x}(t)=f\bigl(x(t),d(t)\bigr),\;t\geq0,\label{eq:nonl_syst}
\end{equation}
where $x(t)\in\mathbb{R}^{n}$ is the state, $d(t)\in\mathbb{R}^{m}$
is the external input, $d\in\mathcal{L}_{\infty}^{m}$, and $f:\mathbb{R}^{n+m}\to\mathbb{R}^{n}$
is a locally Lipschitz continuous function, $f(0,0)=0$. For an initial
condition $x_{0}\in\mathbb{R}^{n}$ and input $d\in\mathcal{L}_{\infty}^{m}$,
define the corresponding solution $x(t,x_{0},d)$ $\forall \: t\geq0$
for which the solution exists.

In this work we will be interested in the following stability properties
\cite{sontag2008}:
\begin{defin}
The system (\ref{eq:nonl_syst}) is called \emph{input-to-state stable}
(ISS), if there are functions $\beta\in\mathcal{K}\mathcal{L}$ and
$\gamma\in\mathcal{K}$ such that 
\[
\Vert x(t,x_{0},d)\Vert\le\beta(\Vert x_{0}\Vert,t)+\gamma(\|d\|_{[0,t)})\quad\forall t\ge0
\]
for any $x_{0}\in\mathbb{R}^{n}$ and $d\in\mathcal{L}_{\infty}^{m}$.
The function $\gamma$ is called the \emph{nonlinear asymptotic gain}.
\end{defin}

\begin{defin}
The system (\ref{eq:nonl_syst}) is called \emph{integral ISS} (iISS),
if there are functions $\alpha\in\mathcal{K}_{\infty}$, $\gamma\in\mathcal{K}$
and $\beta\in\mathcal{K}\mathcal{L}$ such that 
\[
\alpha(\Vert x(t,x_{0},d)\Vert)\le\beta(\Vert x_{0}\Vert,t)+\int\limits _{0}^{t}\gamma(\Vert d(s)\Vert)\,ds\quad\forall t\ge0
\]
for any $x_{0}\in\mathbb{R}^{n}$ and $d\in\mathcal{L}_{\infty}^{m}$.
\end{defin}
Note that the above properties imply that (\ref{eq:nonl_syst}) is
globally asymptotically stable at the origin for zero input $d$; we
will refer to this property as $0$\emph{-GAS}.

These robust stability properties have the following characterizations
in terms of existence of a Lyapunov function: 
\begin{defin}
\label{def:LyapISS_iISS} A smooth $V:\mathbb{R}^{n}\to\mathbb{R}_{+}$
is called \emph{ISS-Lyapunov function} for the system (\ref{eq:nonl_syst})
if there are $\alpha_{1},\alpha_{2},\alpha_{3}\in\mathcal{K}_{\infty}$
and $\eta\in\mathcal{K}$ such that 
\begin{gather*}
\alpha_{1}(\Vert x\Vert)\le V(x)\le\alpha_{2}(\Vert x\Vert),\\
\frac{\partial V(x)}{\partial x} f(x,d)\le\eta(\Vert d\Vert)-\alpha_{3}(\Vert x\Vert)
\end{gather*}
for all $x\in\mathbb{R}^{n}$ and all $d\in\mathbb{R}^{m}$. Such
a $V$ is called \emph{iISS-Lyapunov function} if $\alpha_{3}:\mathbb{R}_{+}\to\mathbb{R}_{+}$
is just a positive definite function.
\end{defin}
\begin{defin}
The system (\ref{eq:nonl_syst}) is called \emph{zero-output smoothly
dissipative} if there is a smooth $V:\mathbb{R}^{n}\to\mathbb{R}_{+}$
with $\alpha_{1},\alpha_{2}\in\mathcal{K}_{\infty}$ and $\chi\in\mathcal{K}$
such that 
\begin{gather*}
\alpha_{1}(\Vert x\Vert)\le V(x)\le\alpha_{2}(\Vert x\Vert),\quad
\frac{\partial V(x)}{\partial x}f(x,d)\le\chi(\Vert d\Vert)
\end{gather*}
for all $x\in\mathbb{R}^{n}$ and all $d\in\mathbb{R}^{m}$.
\end{defin}
The relations between these Lyapunov characterizations and the robust
stability properties are given below:
\begin{theorem}
\label{thm:ISS_Lyap} The system (\ref{eq:nonl_syst}) is ISS if and
only if it admits an ISS-Lyapunov function.
\end{theorem}
The local ISS property can be defined by restricting the domains of
admissible values for $x$ and $d$, then all definitions and the
results save their meanings but locally.
\begin{theorem}
\label{thm:iISS_Lyap} The system (\ref{eq:nonl_syst}) is iISS if
and only if 

$i)$ it admits an iISS-Lyapunov function;

$ii)$ it is $0$-GAS and zero-output smoothly dissipative.
\end{theorem}
Finally, the following property forms a bridge between ISS and iISS
characterizations:
\begin{defin}
The system (\ref{eq:nonl_syst}) is called \emph{strongly iISS}, if
it is iISS and there is $D>0$ such that it is ISS for the inputs
satisfying $\|d\|_{\infty} \leq D$.
\end{defin}

\section{Main results}
\label{sec:33}

Suggesting a max value regularization of \eqref{eq:CLoriginal} yields
\begin{eqnarray}
\dot{x}_{1} & = & x_{2},\label{eq:33closedloop}\\
\dot{x}_{2} & = & -\max\bigl\{\mu,\bigl|x_{1}\bigr|\bigr\}^{-1}\bigl(\gamma x_{1}+\bigl|x_{2}\bigr|x_{2}\bigr)+d,\nonumber 
\end{eqnarray}
where $0<\mu\ll1$ is a parameter, $\gamma>0$ as before is chosen
to compensate the influence of external disturbances $d$, $d\in\mathcal{L}_{\infty}^{1}$
with $\|d\|_{\infty}\leq D$ for a given $D>0$. This system is locally Lipschitz continuous in
$\mathbb{R}^{2}$, hence, it can be easily discretized using a conventional explicit Euler method \cite{Butcher2008}.

To establish the basic stability property for (\ref{eq:33closedloop}),
we first introduce an energy-like Lyapunov function candidate 
\begin{gather*}
E(x)=\gamma z(x_{1})+\frac{1}{2}x_{2}^{2},\\
z(x_{1})=\int_{0}^{x_{1}}\frac{s}{\max\{\mu,|s|\}}ds=\begin{cases}
\frac{x_{1}^{2}}{2\mu} & |x_{1}|<\mu\\
|x_{1}|-\frac{\mu}{2} & |x_{1}|\geq\mu
\end{cases},
\end{gather*}
which is positive definite and radially unbounded, and its time derivative
yields to 
\[
\dot{E}=-\frac{|x_{2}|x_{2}^{2}}{\max\{\mu,\bigl|x_{1}\bigr|\}}+x_{2}d,
\]
which is negative semi-definite if $d\equiv0$. It can be easily shown,
by the LaSalle's invariance principle, that any trajectory of  $x=[x_{1},x_{2}]^{\top}\in\mathbb{R}^{2}$ will not stay in the set 
\[
S=\bigl\{(x_{1},x_{2})\:\bigl|\:x_{2}=0\bigr\},
\]
except the trivial trajectory $x=0$. Therefore, the system (\ref{eq:33closedloop})
is $0$-GAS. Moreover, introducing an extension of the energy-like
function
\[
\mathcal{W}(x)=\ln(1+E(x)),
\]
which is also positive definite and radially unbounded, obtain
\begin{align*}
\dot{\mathcal{W}} & =-\frac{|x_{2}|x_{2}^{2}}{\max\{\mu,\bigl|x_{1}\bigr|\}\bigl(1+E(x)\bigr)}+\frac{x_{2}}{1+E(x)}d  \; \leq \; \frac{1}{\sqrt{2}}|d|,
\end{align*}
thus establishing the zero-output smooth dissipativity of (\ref{eq:33closedloop}).
Therefore, by Theorem \ref{thm:iISS_Lyap}, this system is iISS. Which
is a reasonable conclusion taking into account that the system has
a bounded dissipation in the variable $x_{1}$.

Next, in this section, we will try to demonstrate additional performance
characteristics of this regularized system.

Note that for $|x_{1}|\geq\mu$ we get the original closed-loop dynamics
\eqref{eq:CLoriginal}, hence, we can use the related Lyapunov function
from \cite{ruderman2024robust}, while for $|x_{1}|<\mu$ a simplified
model is obtained. Below considering this system as a switched one,
we will make analysis of both sub-dynamics, first separately, and
next, by uniting these results we provide the conclusion on additional
global properties of (\ref{eq:33closedloop}).

\subsection{Outside $\mu$-region}

\label{sec:33:sub:1}

If $|x_{1}|\geq\mu$, then (\ref{eq:33closedloop}) recovers to the
original closed-loop dynamics \eqref{eq:CLoriginal} from \cite{ruderman2024robust}.
Therefore, the stability and convergence analysis are the same as
provided in \cite{ruderman2024robust}. In particular, the following
Lyapunov function can be applied:
\begin{equation}
V(x)=\gamma|x_{1}|+\frac{1}{2}x_{2}^{2}+\varepsilon\sqrt{|x_{1}|}\text{sign}(x_{1})x_{2},\label{eq:Lyapunov0}
\end{equation}
which is positive definite and radially unbounded for any $\varepsilon\in(0,\sqrt{2\gamma})$
(it can be presented as a quadratic form of $\sqrt{|x_{1}|}\text{sign}(x_{1})$
and $x_{2}$), and it was show in \cite{ruderman2024robust} that
the full time derivative of this Lyapunov function with respect to
the dynamics in \eqref{eq:CLoriginal} admits an estimate:
\begin{gather*}
\begin{align}
\dot{V}\leq & -\left(\varepsilon\left(\gamma-\frac{1}{2}-|d|\right)-\frac{2}{3}|d|^{1.5}\right)\sqrt{|x_{1}|}\\
& -\left(\frac{2}{3}-\varepsilon\right)\frac{|x_{2}|}{|x_{1}|}x_{2}^{2},
\end{align}
\end{gather*}
which for $\gamma>\frac{1}{2}+D+\frac{2}{3\varepsilon}D^{1.5}$ and
$\|d\|_{\infty}\leq D$ guarantees the uniform stability and convergence
in \eqref{eq:CLoriginal} to zero for all $x\in\mathbb{R}^n$.

\subsection{In vicinity to origin}

\label{sec:33:sub:2}

If $|x_{1}|<\mu$, then the corresponding system 
\begin{eqnarray*}
\dot{x}_{1} & = & x_{2},\\
\dot{x}_{2} & = & -\frac{\gamma}{\mu}x_{1}(t)-\frac{1}{\mu}\bigl|x_{2}(t)\bigr|x_{2}(t)+d(t),
\end{eqnarray*}
from (\ref{eq:33closedloop}) can be rewritten in the state-space
form 
\begin{equation}
\dot{x}(t)=\left[\begin{array}{cc}
0 & 1\\
-\dfrac{\gamma}{\mu} & -\dfrac{|x_{2}|}{\mu}
\end{array}\right]\,x(t)+\left[\begin{array}{c}
0\\[1mm]
1
\end{array}\right]\ d(t)\label{eq:33closedloopSS}
\end{equation}
used for the further studies.

For the analysis of (\ref{eq:33closedloopSS}), we first introduce
an energy-like Lyapunov function candidate 
\begin{equation}
\mathcal{E}(x)=\frac{1}{2}\frac{\gamma}{\mu}x_{1}^{2}+\frac{1}{2}x_{2}^{2},\label{eq:Lyapunov1}
\end{equation}
which is positive definite and radially unbounded, and its time derivative
yields to 
\begin{equation}
\dot{\mathcal{E}}=-\frac{|x_{2}|x_{2}^{2}}{\mu}+x_{2}d,\label{eq:dLyapunov1}
\end{equation}
which is negative semi-definite if $d\equiv0$.

\subsubsection{Local ISS property}

In order to obtain the results for $d\ne0$, consider the following
Lyapunov function:
\begin{equation}
\begin{aligned}
W(x)= & \: \mathcal{E}(x)+H(x)+\frac{\mu}{3}\left(\frac{\gamma}{\mu}+1\right)\epsilon_{1}^{4}\mathcal{E}^{1.5}(x) \\
& +\frac{2\sqrt{2}}{7}\epsilon_{2}^{4}\mathcal{E}^{3.5}(x),\label{eq:Lyapunov2}
\end{aligned}
\end{equation}
where $\epsilon_{1},\epsilon_{2}$ are positive parameters whose values
will be specified below, $\mathcal{E}(x)$ is given in (\ref{eq:Lyapunov1}),
and 
\[
H(x)=\frac{1}{4}h^{4}(x), \quad h(x)=\frac{\gamma}{\mu}x_{1}+x_{2}.
\]
Obviously, $W(x)$ is positive definite and radially unbounded since
all terms and parameters are nonnegative, while $\mathcal{E}(x)$
has this property. Note that
\[
\dot{h}=\left(\frac{\gamma}{\mu}+1\right)x_{2}-h-\frac{1}{\mu}|x_{2}|x_{2}+d
\]
by direct computations, and using Young's inequality we can obtain
the inequalities

\[
x_{2}h^{3}\leq\frac{3}{4\epsilon_{1}^{\frac{4}{3}}}h^{4}+\frac{\epsilon_{1}^{4}}{4}x_{2}^{4},\;|x_{2}|x_{2}h^{3}\leq\frac{3}{4\epsilon_{2}^{\frac{4}{3}}}h^{4}+\frac{\epsilon_{2}^{4}}{4}x_{2}^{8}
\]
for any $\epsilon_{1},\epsilon_{2}>0$, leading to
\begin{gather*}
\begin{aligned}
\dot{H} = & \left(\frac{\gamma}{\mu}+1\right)x_{2}h^{3}-h^{4}-\frac{h^{3}}{\mu}|x_{2}|x_{2}+h^{3}d \\
& \leq-  \kappa h^{4}+\left(\frac{\gamma}{\mu}+1\right)\frac{\epsilon_{1}^{4}}{4}x_{2}^{4}+\frac{\epsilon_{2}^{4}}{4\mu}x_{2}^{8}+h^{3}d
\end{aligned}
\end{gather*}
where
\[
\kappa=1-\left(\frac{\gamma}{\mu}+1\right)\frac{3}{4\epsilon_{1}^{\frac{4}{3}}}-\frac{1}{\mu}\frac{3}{4\epsilon_{2}^{\frac{4}{3}}}.
\]
The values of $\epsilon_{1},\epsilon_{2}$ can be chosen big enough
providing that $\kappa>0$. Therefore,

\[
\begin{aligned}\dot{W}\leq & -\frac{1}{\mu}|x_{2}|x_{2}^{2}-\kappa h^{4}\\
 & +\underset{b(x)}{\underbrace{\Biggl[h^{3}+\biggl(1+\frac{\mu}{2}\Bigl[\frac{\gamma}{\mu}+1\Bigr]\epsilon_{1}^{4}\sqrt{\mathcal{E}}+\sqrt{2}\,\epsilon_{2}^{4}\,\mathcal{E}^{2.5}\biggr)x_{2}\Biggr]}}d
\end{aligned}
\]
and the function $b(x)$ cannot be compensated by the negative terms,
but it is bounded by a constant at any bounded vicinity of the origin.
Consequently, according to the result of Theorem \ref{thm:ISS_Lyap},
$W$ is a local ISS-Lyapunov function for (\ref{eq:33closedloopSS}).

\subsubsection{Analysis of linearization}

The system (\ref{eq:33closedloopSS}), when linearized in the equilibrium
$x=0$, has the system matrix 
\begin{equation}
A=\left[\begin{array}{cc}
0 & 1\\[1mm]
-\dfrac{\gamma}{\mu} & 0
\end{array}\right]\label{eq:SysMatrix}
\end{equation}
and constitutes a harmonic oscillator with the angular frequency $\omega_{0}=\sqrt{\frac{\gamma}{\mu}}$.
In order to evaluate the bounds of the control error $x_{1}(t)$ with
$t\rightarrow\infty$ depending on the bounds of the disturbance $\|d\|_{\infty}\leq D$,
consider two scenarios. 

First, consider the case of a constant disturbance
(as it is typical and relevant in several applications), i.e., $d=\bar{d}\in(0,D)$.
The corresponding solution of the forced oscillations with $x(0)=0$
is 
\begin{equation}\label{eq:forcedOscillations1}
x_{1}(t) = \frac{\mu}{\gamma}\bar{d}-\frac{\mu}{\gamma}\bar{d}\cos(\omega_{0}t), \quad x_{2}(t) = \frac{\omega_0 \mu }{\gamma}\bar{d}\sin(\omega_{0}t).
\end{equation}
From \eqref{eq:forcedOscillations1}, it is evident
that at any sufficiently small time instant
$t^{\ast}>0$, the forced solution $x_{1}(t^{\ast})$ and so its time
derivative $x_{2}(t^{\ast})$, cf. \eqref{eq:forcedOscillations1}, become non zero. That leads to a local linearized in $\bigl[x_1(t^{\ast}),x_2(t^{\ast})\bigr]$ from \eqref{eq:33closedloop} model
\begin{equation}
\ddot{x}_{1}(t)+\sigma^{\ast}\dot{x}_{1}(t)+\omega^{2}x_{1}(t)=\bar{d},\label{eq:locallinear1}
\end{equation}
with a positive damping factor $\sigma^{\ast}=|x_{2}(t^{\ast})|\mu^{-1} > 0$.
Applying the final value theorem to (\ref{eq:locallinear1}), either
in Laplace domain or time domain i.e. $\dot{x}=0$, results in (cf.
with (\ref{eq:forcedOscillations1})) 
\begin{equation}
x_{1}(t)\rightarrow\frac{\mu}{\gamma}\,\bar{d}\quad\hbox{ for }\quad t\rightarrow\infty.\label{eq:errorBoundConst}
\end{equation}

Next, it can be shown that for both the harmonic oscillator (\ref{eq:33closedloopSS})
with (\ref{eq:SysMatrix}) and its damped counterpart on the 
left-hand-side of (\ref{eq:locallinear1}), the largest excitation and so the residual control
error $\bigl|x_{1}(t)\bigr|_{t\rightarrow\infty}$ appear when $d(t)$
meets the natural frequency $\omega_{0}$. Thus, consider (again) the local
linear model 
\begin{equation}
\ddot{x}_{1}(t)+\sigma^{\ast}\dot{x}_{1}(t)+\omega_{0}^{2}x_{1}(t)=\tilde{d}\cos(\omega_{0}t)\label{eq:forcedOscillator}
\end{equation}
with the disturbance amplitude $\tilde{d}<D$. Neglecting the homogeneous
part of the solution of (\ref{eq:forcedOscillator}), which is always
converging to zero $\forall\,\sigma^{\ast}>0$, in other words assuming without
loss of generality $x_{1}(0)=0,\,\dot{x}_{1}(0)=0$, the non-homogeneous
solution of (\ref{eq:forcedOscillator}) is 
\begin{equation}
x_{1}=\frac{\tilde{d}}{\sigma^{\ast}\omega_{0}}\bigg(\sin(\omega_{0}t)-\dfrac{e^{-0.5\sigma^{\ast} t}\,\sin\bigl(\sqrt{\omega_{0}^{2}-(0.5\sigma^{\ast})^{2}}\,t\bigr)}{\sqrt{\omega_{0}^{2}-(0.5\sigma^{\ast})^{2}}}\biggr).\label{eq:forcedSolution}
\end{equation}
One can recognize that the second term in brackets of (\ref{eq:forcedSolution})
is vanishing as the time increases, so that the steady-state value
results in 
\begin{equation}
\bar{x}_{1}=x_{1}(t)\bigr|_{t\rightarrow\infty}=\frac{\tilde{d}}{\sigma^{\ast}\omega_{0}}\sin(\omega_{0}t).\label{eq:forcedSS}
\end{equation}
The corresponding time derivative is 
\begin{equation}
\bar{x}_{2}=\dot{\bar{x}}_{1}=\frac{\tilde{d}}{\sigma^{\ast}}\cos(\omega_{0}t).\label{eq:forcedSSx2}
\end{equation}
From (\ref{eq:forcedSSx2}) and $\sigma^{\ast}=\bigl|\max(\bar{x}_{2})\bigr|\mu^{-1}$ one obtains 
\[
\max(\bar{x}_{2})=\frac{\tilde{d}\,\mu}{\max(\bar{x}_{2})},
\]
that leads to 
\begin{equation}
\max\bigl(\bar{x}_{2}\bigr)=\sqrt{\tilde{d}\,\mu}.\label{eq:Maxx2}
\end{equation}
Since for a forced harmonic oscillator \eqref{eq:forcedOscillator} in steady-state, the maximal
value of a periodic $x_2(t)$ leads to the correspondingly maximal $|x_{1}|$,
one can obtain from (\ref{eq:forcedSS}) and (\ref{eq:Maxx2}) 
\begin{equation}
\max|x_{1}|=\frac{\tilde{d}\,\mu}{\sqrt{\gamma\tilde{d}}}.\label{eq:Maxx1}
\end{equation}
The estimate (\ref{eq:Maxx1}) constitutes the upper bound of the
control error $x_{1}(t)$ for a bounded disturbance $|d(t)|\leq\tilde{d}<D$.

The following numerical results confirms the estimated upper bounds
\eqref{eq:errorBoundConst} and \eqref{eq:Maxx1}. The closed-loop
control system \eqref{eq:33closedloop} is simulated with use of the
first-order Euler solver and fixed step sampling $0.00001$ sec, while
the control gain is assigned to $\gamma=100$ and the initial conditions
to $x(0)=[1,\,0]^{\top}$. 
\begin{figure}[!h]
\centering 
\includegraphics[width=0.98\columnwidth]{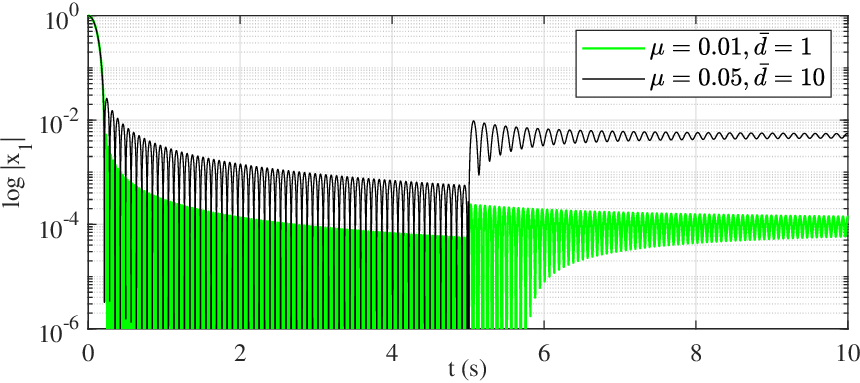}
\includegraphics[width=0.98\columnwidth]{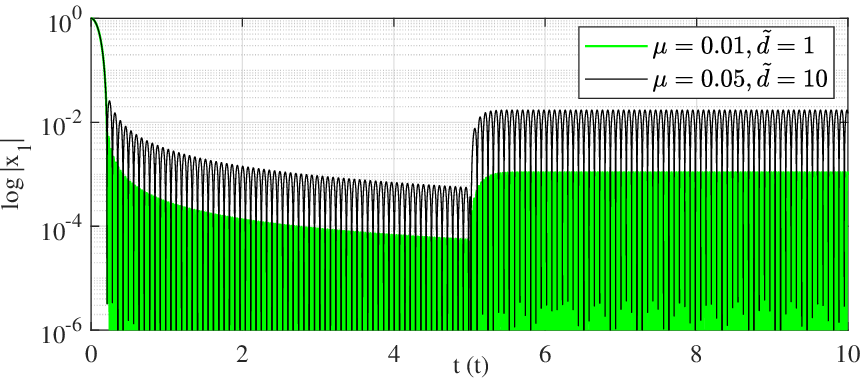}
\caption{Output convergence (on logarithmic scale) of the system \eqref{eq:33closedloop}: above -- for two pairs $\mu=\{0.01,\,0.05\}$ and $\bar{d}=\{1,\,10\}$ with constant disturbance applied at $t=5$ sec; below -- for two pairs $\mu=\{0.01,\,0.05\}$
and $\tilde{d}=\{1,\,10\}$ with harmonic disturbance applied at $t=5$ sec.}
\label{fig:simConstDist} 
\end{figure}
Two type of disturbances, the constant one $d(t)=\bar{d}$ and the (resonant) harmonic one $d(t)=\tilde{d}\cos\bigl(\sqrt{\gamma/\mu}\cdot t\bigr)$ are assumed and applied at time $t=5$ sec to the system \eqref{eq:33closedloop}. Two pairs of the parameter values are shown for comparison in both cases, $\mu=\{0.01,\,0.05\}$ and
$\bar{d}=\tilde{d}=\{1,\,10\}$. The convergence of the output state
absolute value are shown logarithmically in Fig. \ref{fig:simConstDist}, for the constant disturbance
above and for the (resonant) harmonic disturbance below. Both upped bounds, correspondingly
final values, coincide exactly with those computed by \eqref{eq:errorBoundConst}
and \eqref{eq:Maxx1}.

\subsection{Local ISS for (\ref{eq:33closedloop})}

As it has been established previously, the system (\ref{eq:33closedloop})
is iISS, which implies global boundedness and convergence of trajectories
in the presence of disturbances $d$ with a properly bounded integral.
However, if $d\in\mathcal{L}_{\infty}^{1}$, then the trajectories
may be unbounded. Also, it has been shown above that in a neighborhood
of the origin, with $|x_{1}|<\mu$, the system admits a local ISS-Lyapunov
function $W$ given in (\ref{eq:Lyapunov2}), and for $|x_{1}|\geq\mu$
the system has a strictly decaying Lyapunov function $V$ presented
in (\ref{eq:Lyapunov0}). Then, we arrive at our main result by utilizing
the combination of $W$ and $V$:
\begin{theorem}
For any choice of $\gamma>\max\{4,2D+4\sqrt{2}D^{1.5}\}$ and $\mu>0$,
with a sufficiently small $D$ (for the chosen value of $\mu$), the system (\ref{eq:33closedloop})
with the disturbances satisfying $\|d\|_{\infty}\leq D$ is strongly
iISS.
\end{theorem}
\begin{proof}
Our strategy to design the required ISS-Lyapunov function consists,
first, in slight modification of one presented in (\ref{eq:Lyapunov0})
by replacing $|x_{1}|$ (integral of nonlinearity in \eqref{eq:CLoriginal})
by $z(x_{1})$ (integral of the counterpart in (\ref{eq:33closedloop})):
\[
\tilde{V}(x)=\gamma z(x_{1})+\frac{1}{2}x_{2}^{2}+\varepsilon\sqrt{z(x_{1})}\text{sign}(x_{1})x_{2},
\]
which is again positive definite and radially unbounded for $\varepsilon\in(0,\sqrt{2\gamma})$.
The time derivative of $\tilde{V}$ on the trajectories of (\ref{eq:33closedloop})
can be written as follows:
\begin{gather*}
\dot{\tilde{V}}=-\frac{|x_{2}|x_{2}^{2}}{\max\{\mu,|x_{1}|\}}-\gamma\varepsilon\frac{\sqrt{z(x_{1})}|x_{1}|}{\max\{\mu,|x_{1}|\}}\\
+\varepsilon\left(\frac{1}{2}-\frac{z(x_{1})}{|x_{1}|}\text{sign}(x_{1})\text{sign}(x_{2})\right)\frac{|x_{1}|}{\sqrt{z(x_{1})}}\frac{x_{2}^{2}}{\max\{\mu,|x_{1}|\}}\\
+(\varepsilon\sqrt{z(x_{1})}\text{sign}(x_{1})+x_{2})d\\
\leq-\frac{|x_{2}|x_{2}^{2}}{\max\{\mu,|x_{1}|\}}-\gamma\varepsilon\frac{\sqrt{z(x_{1})}|x_{1}|}{\max\{\mu,|x_{1}|\}}\\
+\varepsilon\left(\frac{1}{2}+\frac{z(x_{1})}{|x_{1}|}\right)\frac{\frac{|x_{1}|}{\sqrt{z(x_{1})}}x_{2}^{2}}{\max\{\mu,|x_{1}|\}}+(\varepsilon\sqrt{z(x_{1})}+|x_{2}|)|d|.
\end{gather*}
Note that $\frac{z(x_{1})}{|x_{1}|}\leq1$, and
\[
\frac{|x_{1}|}{\sqrt{z(x_{1})}}x_{2}^{2}\leq\frac{2}{3}|x_{2}|^{3}+\frac{1}{3}\left(\frac{|x_{1}|}{\sqrt{z(x_{1})}}\right)^{3}
\]
following the Young's inequality, then we obtain:
\begin{gather*}
\dot{\tilde{V}}\leq(\varepsilon\sqrt{z(x_{1})}+|x_{2}|)|d|-(1-\varepsilon)\frac{|x_{2}|x_{2}^{2}}{\max\{\mu,|x_{1}|\}}\\
-\varepsilon\frac{\left(\gamma\sqrt{z(x_{1})}|x_{1}|-\frac{1}{2}\left(\frac{|x_{1}|}{\sqrt{z(x_{1})}}\right)^{3}\right)}{\max\{\mu,|x_{1}|\}}.
\end{gather*}
Using again the Young's inequality and recalling $|d|\leq D$:
\begin{align*}
|x_{2}||d| & =\frac{|x_{2}|}{\sqrt[3]{\max\{\mu,|x_{1}|\}}}\sqrt[3]{\max\{\mu,|x_{1}|\}}|d|\\
 & \leq\frac{1}{3}\frac{|x_{2}|^{3}}{\max\{\mu,|x_{1}|\}}+\frac{2}{3}\sqrt{\max\{\mu,|x_{1}|\}}D^{1.5},
\end{align*}
the upper estimate for the derivative of the Lyapunov function can
be further simplified:
\begin{gather*}
\dot{\tilde{V}}\leq-\varepsilon k(x_{1})-\left(\frac{2}{3}-\varepsilon\right)\frac{|x_{2}|x_{2}^{2}}{\max\{\mu,|x_{1}|\}} -\varepsilon e(x_{1},D),
\end{gather*}
where
\begin{gather*}
k(x_{1})=\frac{\gamma\sqrt{z(x_{1})}|x_{1}|-\left(\frac{|x_{1}|}{\sqrt{z(x_{1})}}\right)^{3}}{2\max\{\mu,|x_{1}|\}}\\
=\begin{cases}
\gamma\frac{1}{\left(\sqrt{2\mu}\right)^{3}}x_{1}^{2}-\sqrt{2\mu} & |x_{1}|<\mu\\
\frac{\gamma\sqrt{|x_{1}|-\frac{\mu}{2}}|x_{1}|-\left(\frac{|x_{1}|}{\sqrt{|x_{1}|-\frac{\mu}{2}}}\right)^{3}}{2|x_{1}|} & |x_{1}|\geq\mu
\end{cases},\\
e(s,D)=\frac{\sqrt{z(s)}}{\max\{\mu,|s|\}}\left(\frac{\gamma}{2}|s|-D-\frac{\max\{\mu,|s|\}^{\frac{3}{2}}}{\frac{3}{2}\varepsilon\sqrt{z(s)}}D^{\frac{3}{2}}\right).
\end{gather*}
First, analyzing the expression of $k(x_{1})$ we conclude that if
$\gamma>4$, then $k(x_{1})>0$ for $|x_{1}|\geq\mu$ and, while for
$|x_{1}|<\mu$ the constant negative term is obviously dominating
for $|x_{1}|\leq\frac{2\mu}{\sqrt{\gamma}}$, and increasing the gain
$\gamma$ this domain can be narrowed. Performing similar analysis
we get for $|x_{1}|\geq\mu$:
\begin{align*}
e(x_{1},D) & =\sqrt{|x_{1}|-\frac{\mu}{2}}\left(\frac{\gamma}{2}-D-\frac{2\sqrt{|x_{1}|}}{3\varepsilon\sqrt{|x_{1}|-\frac{\mu}{2}}}D^{1.5}\right)\\
 & \geq\sqrt{|x_{1}|-\frac{\mu}{2}}\left(\frac{\gamma}{2}-D-\frac{2\sqrt{2}}{3\varepsilon}D^{1.5}\right)
\end{align*}
that is strictly positive provided that 
\[
\frac{\gamma}{2}>D+\frac{2\sqrt{2}}{3\varepsilon}D^{1.5},
\]
and for $|x_{1}|<\mu$:
\begin{align*}
e(x_{1},D) & =\frac{\frac{\gamma}{2}|x_{1}|}{\mu}\sqrt{\frac{x_{1}^{2}}{2\mu}}-\sqrt{\frac{x_{1}^{2}}{2\mu}}D-\frac{2\sqrt{\mu}}{3\varepsilon}D^{1.5}\\
 & \geq-\sqrt{\frac{\mu}{2}}D-\frac{2\sqrt{\mu}}{3\varepsilon}D^{1.5}
\end{align*}
having the disturbance gain of order $\sqrt{\mu}$. Note that for
$\gamma>4$, $\varepsilon=\frac{1}{3}$ is an admissible choice, giving
the restriction stated in the formulation of the theorem, then $\dot{\tilde{V}}$
is strictly negative for $|x_{1}|\geq\mu$, and locally close to the
origin, with $|x_{1}|<\mu$, a bias and influence of the disturbances
appear. Therefore, $\tilde{V}$ is a practical ISS-Lyapunov function
for (\ref{eq:33closedloop}) with $\|d\|_{\infty}\leq D$. In order
to construct an ISS-Lyapunov function, consider the following candidate:
\[
U(x)=\tilde{V}^{3}(x)+\sigma(W(x)),
\]
where $\sigma\in\mathcal{K}$ is a bounded continuously differentiable
function which is reduced to a linear map close to the origin (it is designed in a way 
to guarantee that $\sigma(W(x)))=\text{const}$ for $|x_1|\geq\mu$). The
idea behind this design is as follows. First, clearly, $U$ is positive
definite and radially unbounded due to $\tilde{V}$ possesses these
properties. Second, the derivative of the term $\tilde{V}^{3}(x)$
has the form $3\tilde{V}^{2}\dot{\tilde{V}}$, and it is strictly
negative for $|x_{1}|\geq\mu$, and where $\dot{U}=3\tilde{V}^{2}\dot{\tilde{V}}$
due to a choice of $\sigma$, while for $|x_{1}|<\mu$ the bias is
now multiplied by the term of orders $x_{1}^{4}$ and $x_{2}^{4}$
which can be compensated by the negative terms of order $x_{1}^{4}$
and $|x_{2}|x_{2}^{2}$ contained in $\dot{W}$, provided that the
disturbances are sufficiently small. Hence, by a proper weighting
of $\tilde{V}$ and $W$, we can guarantee that $U$ is an ISS-Lyapunov
function for (\ref{eq:33closedloop}) with $\|d\|_{\infty}\leq D$.
\end{proof}
Remarkable is that the requirements imposed here on
the gain $\gamma$ are more restrictive than for \eqref{eq:CLoriginal}
given in \cite{ruderman2024robust}:
\[
\gamma>\frac{1}{2}+D+2D^{1.5},
\]
which is however of the same shape; and that is the price to pay for
the regularization in vicinity to the origin.

\subsection{Alternative regularization}
\label{sec:altreg}

For the non-overshooting quasi-continuous sliding mode controller
\cite{ruderman2024robust}, we also suggest the control regularization
scheme, similar to that used in \cite{ruderman2021convergent}, 
\begin{equation}
u=-\bigl(|x_{1}|+\mu \bigr)^{-1} \bigl(\gamma x_{1} + |x_{2}|x_{2} \bigr),\label{eq:regcontrol}
\end{equation}
where the regularization factor $0<\mu\ll1$ is the second design
parameter, in addition to the control gain $\gamma>0$. The resulting
closed-loop control system, with a bounded perturbation $\|d\|_{\infty}\leq D$
for $D>0$, yields
\begin{eqnarray}
\dot{x}_1 & = & x_2,\label{eq:closedloop}\\
\dot{x}_2 & = & -\bigl(|x_{1}|+\mu \bigr)^{-1}\bigl(\gamma x_{1}+|x_{2}|x_{2}\bigr)+d,\nonumber 
\end{eqnarray}
where a similar to \eqref{eq:33closedloop} notation can be applied.

An energy-like Lyapunov function candidate is
\begin{equation}
E(x)=\gamma\Bigl(|x_{1}|-\mu\ln\bigl(\mu+|x_{1}|\bigr)\Bigr)+\frac{1}{2}x_{2}^{2},\label{eq:lyap1}
\end{equation}
resulting in 
\begin{equation}
\dot{E}=-\frac{|x_{2}|x_{2}^{2}}{\mu+|x_{1}|}+x_{2}d.\label{eq:Dlyap1}
\end{equation}
Next, an analysis similar to the one provided in the previous section
can be repeated to this regularized system. This would, however, go beyond the scope of the present work and will be therefore omitted for the sake of brevity. 

\section{Numerical comparison}
\label{sec:3} 

\begin{figure}[!h]
\centering 
\includegraphics[width=0.95\columnwidth]{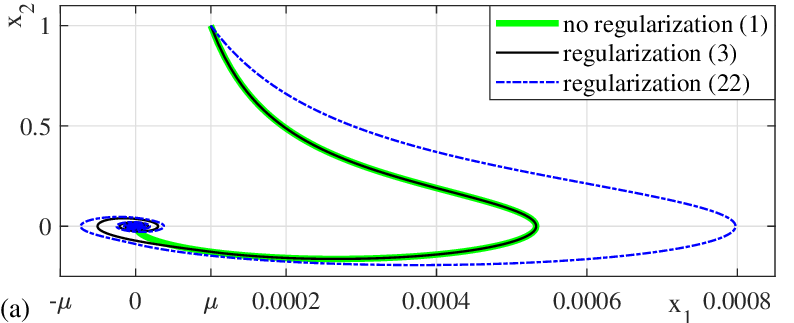}
\includegraphics[width=0.95\columnwidth]{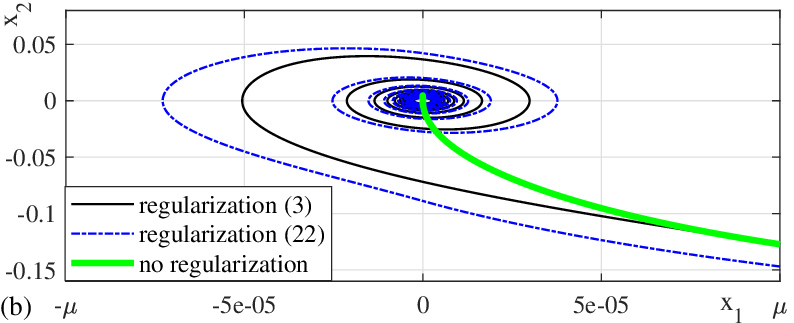}
\includegraphics[width=0.95\columnwidth]{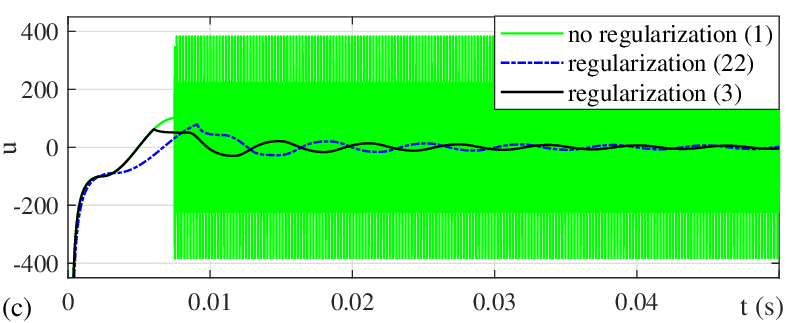}
\caption{Phase portrait of the regularized and not regularized closed-loop control systems: convergence of trajectories with $x(0)=[0.0001,1]^{\top}$ in (a), zoom in the $[-\mu, \mu]$ vicinity to origin in (b), and control signal in (c).}
\label{fig:sim2} 
\end{figure}
The numerical simulations with the first-order Euler solver and the fixed step sampling of $0.00001$ sec, the assigned control gain $\gamma=100$ and regularization factor $\mu = 0.0001$, and the initial conditions $x(0)=[0.0001,1]^{\top}$ are shown in Fig. \ref{fig:sim2}. Three unperturbed (i.e., $d=0$) closed-loop control systems, the original \eqref{eq:CLoriginal} and the regularized \eqref{eq:33closedloop} and \eqref{eq:closedloop}, are compared to each other in terms of the trajectories phase-portrait. The zoom in the $[-\mu, \mu]$ vicinity to the origin is shown in Fig. \ref{fig:sim2} (b) for the sake of a better visualization. The chattering suppression at the equilibrium, achieved by both regularization schemes, is shown in Fig. \ref{fig:sim2} (c), while $u(t)$ converges asymptotically to zero over the time. Note that in a perturbed system case (i.e. $d \neq 0$), the control $u(t)$ converges towards $d(t)$, cf. \cite[Figure~4]{ruderman2024robust}.

\section{Conclusions}
\label{sec:4}

Two regularization schemes for the control \cite{ruderman2024robust}, which is discontinuous in the stable origin only, were provided in the paper. The regularization preserves ISS and finite-time convergence properties outside a close vicinity (given by $x_1 \in [-\mu, \, \mu]$) of the origin, while providing iISS property for a sufficiently small $\mu \ll 1$ region and upper bounded unknown perturbations. A linearized equivalent system dynamics was used to estimate the residual control error in dependency of the disturbance upper bound and control design parameters. 
Future works might be concerned with convergence time estimation for the regularized control \cite{ruderman2024robust}.

\bibliographystyle{IEEEtran}
\bibliography{references_}

\end{document}